\documentclass[11pt]{article}
\usepackage[onehalfspacing]{setspace}
\usepackage{fullpage}
\usepackage{comment}
\usepackage{todonotes}
\usepackage{soul}

\usepackage{dsfont}

\usepackage{psfrag}

\usepackage{calrsfs}
\usepackage{mathrsfs}
\usepackage{pdfsync}
\usepackage{amsmath, amsthm, amssymb, amsfonts}
\usepackage[shortlabels]{enumitem}
\usepackage{url}
\usepackage{float}

\usepackage{tikz}

\usetikzlibrary{arrows,decorations.pathmorphing,backgrounds,positioning,fit,automata}

\usetikzlibrary{arrows}
\usetikzlibrary{petri}
\usetikzlibrary{topaths}

\usepackage{indentfirst,calc,euscript}
\usepackage{setspace}
\usepackage[reals]{layout}
\usepackage{xr}
\usepackage{amscd}

\usepackage{sgame}
\usepackage{subfigure}

\usepackage[colorlinks=true,breaklinks=true,bookmarks=true,urlcolor=blue,
     citecolor=blue,linkcolor=blue,bookmarksopen=false,draft=false]{hyperref}

\renewcommand{\div}{\operatorname{div}}

\newcommand{\ignore}[1]{}

\newtheorem{theorem}{Theorem}

\newtheorem{lemma}[theorem]{Lemma}

\theoremstyle{definition}

\newtheorem{definition}[theorem]{Definition}

\newtheorem{remark}[theorem]{Remark}

\numberwithin{equation}{section}

\numberwithin{theorem}{section}

\newcommand{\m}{\mathbb}

\newcommand{\bbr}{\mathbb{R}}
\newcommand{\bbT}{\mathbb{T}}
\newcommand{\bbe}{\mathbb{E}}
\newcommand{\caln}{\mathcal{N}}
\newcommand{\calp}{\mathcal{P}}
\newcommand{\call}{\mathcal{L}}
\newcommand{\td}{\text{d}}
\newcommand{\diff}{\ \td}

\usepackage{faktor}
\usepackage{xfrac}

\author{
    Eran Shmaya \thanks{Stony Brook University, United States}
    \and 
    Bruno Ziliotto\thanks{CEREMADE, CNRS, PSL Research Institute, Paris Dauphine University, France.}
}

\title{Bayesian Learning in Mean Field Games}
\begin{document}
\maketitle
\bibliographystyle{plain}
\begin{abstract}
%We consider a mean-field game model where the cost functions depend on a fixed parameter, called \textit{state}, that is unknown to players and normally distributed. Players receive a stream of normally distributed private signals on the state along the game. We derive a mean field system satisfied by the equilibrium payoff of the game with fixed duration, and prove existence of a solution under standard regularity assumptions. Assuming that for each state, the coupling function satisfies the monotonicity assumption of Lasry and Lions, uniqueness of the solution is proved. 
We consider a mean-field game model where the cost functions depend on a fixed parameter, called \textit{state}, which is unknown to players. Players learn about the state from a a stream of private signals they receive throughout the game. We derive a mean field system satisfied by the equilibrium payoff of the game and prove existence of a solution under standard regularity assumptions. Additionally, we establish the uniqueness of the solution when the cost function satisfies the monotonicity assumption of Lasry and Lions at each state.
\end{abstract}
\section*{Introduction}
Mean field games (MFGs, see \cite{LL06a,LL06b,LL07} and \cite{HMC06}) feature a continuum of non-atomic players who control their own positions, and optimize a cost function that depends on their position, their speed, and the distribution positions of the other players. These games have been intensively studied in the last two decades (see e.g. \cite{CF18,CDL19}), and found a wide range of applications. In most of these works, players are assumed to know perfectly the variables determining the cost function. This is in contrast with Game Theory literature, where models with incomplete information have been intensively studied in a wide variety of frameworks (see e.g. \cite{AM95}, \cite[Chapter III]{MSZ}, \cite[Chapter 10]{MSZ2} and \cite{forges92}). Incomplete information is typically on some variable affecting the cost function and the dynamics. Such a variable may be static or dynamic, and players may acquire information on it through the observation of a stream of private or public signals, or through the observation of the other's actions. As far as incomplete information is concerned, MFG literature is rather scarce. Models where players are uninformed of their own state have been studied in \cite{SC16,SC19,FC20,BRT23}, where players do not know the average position of the other players \cite{CJ18,bertucci22}, and where one major player has a private 
information and discloses it strategically through her control to a population of small players \cite{BCR23}. 
%An exciting line of research is to develop new models of games with incomplete information, based nably   \textcolor{red}{talk about MFG learning literature?} 
\paragraph{Contribution of the paper}
In this paper, we introduce a MFG model where cost functions depend on a fixed parameter called \textit{state}. Players do not know the exact value of the state, and receive a stream of private signals along the game. A major issue is the procedure by which the players learn the state of nature from their signals. We adopt the Bayesian inference approach, which is the standard approach in game theory: the players update their beliefs about the state using Bayes' formula. 
In a contemporary paper, \cite{BCR23} also study a MFG model with an unknown state. A major difference between their model and ours is that in their model all players observe the same signal, whereas in our model players observe idiosyncratic signals, conditionally independent on the states. The model in~\cite{BCR23} reduces to MFG with common noise. In contrast, our model does not reduce to MFG with common noise and a different system of equations is needed to describe the equilibrium. 

Incorporating Bayesian inference in mean field games introduces a challenge because the posterior beliefs about the parameter  live in an infinite dimensional space. The work \cite{BCR23}, following \cite{AM95}, assumes a finite state space, so beliefs are in a finite dimensional simplex. Our approach is different. The tractability of our model relies on the assumption that the state and the signals are jointly normally distributed, so all posteriors beliefs are also normal, and are determined by the conditional expectation of the states. Moreover, our signaling process is such that future information depends on the state alone and not on current information. This assumption on the signal process makes the posterior belief a Markov process and simplifies Bayesian updating. 

We derive a PDE system that describes the equilibrium solution of the game, composed with a Hamilton-Jacobi equation and a family of Fokker-Planck equations, indexed by the state. The space variable of each player in this system is a pair $(z,x)$, where $x$ represents the position and $z$ is a variable that aggregates the information of the player about the state. Such a system is not a classic MFG system. Indeed, the state appears explicitly as a parameter in the Fokker-Planck equation, while in the Hamilton-Jacobi equation, it appears only through the belief of players. This reflects the fact that the signal dynamics is driven by the state, while players controls are driven by the information variable $z$. Hence, one can not apply directly existence and uniqueness result to this system. Under standard regularity assumptions, we prove existence of a solution, using the classic method of Shauder fixed-point theorem. 

We then turn to the question of the uniqueness of the solution of our system. A well-known condition under which a standard MFG system has a unique solution is the monotonicity condition of Lions and Lasry \cite{LL07}, which can be intuitively interpreted by the fact that players have a preference for being far away from each other. Assuming that the monotonicity condition is satisfied ``state-by state'', we prove uniqueness of our system. This is rather surprising, since the state is not observed by players; one may have expected instead a monotonicity condition expressed in terms of the distribution of the state given players information, as it is the case for instance in \cite{bertucci22}, in a different setting. %The proof uses crucially the fact that the state and signals are normally distributed. 
%provided Moreover, we prove uniqueness of the solution, under an assumption that can be viewed as a parametrized version of the classic monotonicity assumption \cite{LL07}. It is rather surprising that such a monotonicity assumption can be expressed explicitly in terms of the state, since the latter is not observed by players. Instead, one may have expected an assumption involving  (\textcolor{red}{detail more}). 
%The "state of nature" in game theory refer to the different possible configurations or scenarios that encapsulate the players' incomplete information. Each player has some information 
\paragraph{Related game-theoretic models}
The mean-field game literature parallels the game theoretic literature on large games with anonymous players. Here anonymity means that each players' payoff depends on opponents' actions and types only through their empirical distribution, but not on the identity of the opponent that played each action. The literature starts with static games that model one-shot interaction. A canonical model in this literature, the so called \emph{continuum of players} or \emph{non-atomic} approach~\cite{schmeidler1973equilibrium,mas1984theorem}, assumes that idiosyncratic randomizations corresponding to different players actions completely wash out in the aggregate, so that the outcome of the game is fully deterministic. The non-atomic model can also allow some idiosyncratic characteristics of the players, or types in the game-theoretic terminology, again under the assumption that there is no aggregate uncertainty about the distribution of types.  There are many results that relate the asymptotic outcomes of anonymous games with $N$ players, where $N$ goes to infinity with the outcome of the corresponding non-atomic model (see e.g. ~\cite{carmona2022approximation}).

A seminal paper of Green~\cite{green1984continuum} combined large games with (discrete-time) dynamic games. Green's paper and most of its follow-up study games with perfect information, in which the fundamentals of the game are known to the players. Green studied repeated games, in which the players play the same game at every period. Jovanovic~\cite{jovanovic1988anonymous}, Bergin and Bernhadt~\cite{bergin1992anonymous}, and Hoppenhayn~\cite{hopenhayn1992entry} studied large games in which the payoff depends on the history of the game and stochastic games with some aggregated uncertainty.

There is a huge literature on Bayesian learning in game theory (See Young~\cite[Chapter 7]{young2004strategic}). Most of these papers are about games with finite number of players. 
Kalai and Shmaya~\cite{kalai2018large} seem to be the first paper that studies dynamic non-atomic games in which the fundamental of the game are not known initially. Their model also includes Bayesian learning of the fundamentals. In addition to the fact that their model is in discrete time, a major difference between the model in this paper and their model is that the signals in Kalai and Shmaya's paper are commonly observed by all players so at every period all players hold the same belief about the state of nature, whereas in our paper beliefs are idiosyncratic. 
\paragraph{Paper outline}
%The rest of the paper is organized as follows. 
Section \ref{sec:model} describes the game model, while Section \ref{sec:system} establishes the corresponding MFG system and relates it to the state of the art. Section \ref{sec:existence} proves existence of a classical solution of the MFG system, under standard regularity assumptions. Section \ref{sec:uniqueness} proves uniqueness under a monotonicity assumption. 
\paragraph{Notations}

Throughout the paper, the notation $\caln(m,v)$ stands for the normal distribution with mean $m$ and variance $v$. The $n$-dimensional torus, defined as the quotient of $\m{R}^n$ by $\m{Z}^n$, is denoted by  $\m{T}^n$. All finite-dimensional spaces under consideration will be equipped with the Borelian $\sigma$-algebra. Given a measurable set $A$, the notation $\mathcal{P}(A)$ stands for the probability distributions over $A$.       %:=\displaystyle \sfrac{\m{R}^n}{\m{Z}^n}$ 
\section{Model} \label{sec:model}
\subsection{A control problem}\label{se:control}
The MFG system considered in this paper is interpreted as a game with a continuum of players, where each player faces the control problem described by the following elements:
\begin{itemize}
\item
A random \emph{state} $S$ in $\bbr$, 
  unknown to the player, with distribution $\caln(0,1)$. 
\item
A signalling process $(Z_t)$ for the player, that is a process on $\bbr$ that follows the stochastic differential equation (SDE)
\begin{equation}\label{dzt}
dZ_t=Sdt+\sigma dB_t,
\end{equation}
with $\sigma > 0$.
\item
A position $X_t$ lying in the $n$-dimensional torus $\m{T}^n$, whose trajectory is controlled by the player, according to the equation
\begin{equation}\label{dxt}
dX_t=\alpha_t dt+\sigma'dB'_t,
\end{equation}
where $(B_t)$ and $(B'_t)$ are independent Brownians, and the control variable $\alpha=(\alpha_t)_{t\ge 0}$ is measurable with respect to the filtration generated by $(X_t,Z_t)$. When we want to emphasize the dependence of the player's signal and position process on $\alpha$ we also use the notation $(X^\alpha_t,Z^\alpha_t)$. The position space is taken as the torus for simplification purpose. 
\item
A {\emph{position distribution}}  $\rho_{s,t}\in \calp(\bbT^n)$, with the interpretation that $\rho_{s,t}$ is the distribution density over positions at time $t$, conditional on the state being $s$, 
\item
A \textit{total cost} that is a function of the control $(\alpha_t)_{t\ge 0}$ used by the player, defined by

\begin{equation}\label{cost}
\m{E} \left( \int_0^T C(S,X_t,\alpha_t)+F(S,X_t,\rho_{S,t}) \diff t \right) + G(S,X_T,\rho_{S,T}),
\end{equation}
where the \emph{cost functions} $C$,$F$, and $G$ satisfy some regularity assumption that will be specified in the next section,  
and the expectation is taken over the random state $S$ and the random path $X_t$ of the player.
The additive structure of the flow cost, with one component that does not depend on the population distribution, and a second component that does not depend on the control, is standard.
Compared to the classic MFG model, the new feature is that the cost functions also depend on the unknown state. 

We assume that the flow cost is not observed by the player, so that all information about the state comes from the signaling process. %The fact that the player does not observe their flow cost is a rather common feature of games with incomplete information (see e.g. REF). 
Section \ref{sec:app} provides examples where such a situation arises. The fact that players do not observe their flow cost is a common feature of several classic incomplete information models, such as Partially Observable Markov Decision Processes and stochastic games with signals \cite{MSZ}. 
 %\textcolor{red}{Dire pourquoi cela fait sens que les joueurs n'observent pas leur coût}
\end{itemize}
\subsection{Equilibrium}
A \textit{game distribution} is a mapping $m=m:\m{R}\times [0,T] \times \m{T}^n\times \m{R}   \rightarrow \m{R}$ such that $m_{s,t}$ is a probability density over $\m{T}^n \times \m{R}$ for every $s,t$. The game distribution will have two related interpretations:
\begin{itemize}
\item  $m_{s,t}$ is the empirical distribution over positions and signals at time $t$ if the state is $s$. 
\item $m_{s,t}$ is the  law of distribution of the (random) position and signal of a player at time $t$ if the state is $s$.  
\end{itemize}
%Given a game distribution $m=(m_{s,t})$, denote by $\hat{m}$ the distribution such that $\hat{m}_{s,t}$ is the marginal of $m_{s,t}$ over positions. 
%Let $\hat m_0$ be an initial player density over $\m{R}^n$, and let $m_0=\hat m_0\times \delta_0$. 
The distribution $m_{s,t}$ is an 
\textit{equilibrium} if there exists a control $\alpha$ such that:\begin{itemize}\item $m_{s,t}=\call(X_{s,t}^\alpha, Z_{s,t}^\alpha)$ is the probability density function of $X_{s,t},Z_{s,t}$.\item %$\alpha$ minimizes the single player cost~\eqref{cost} with $\rho_{s,t}=\hat m_{s,t}$.
$\alpha$ minimizes the single player cost~\eqref{cost} where $\rho_{s,t}$ is the marginal of $m_{s,t}$ over positions.\end{itemize}
The first item captures the assumption that, while each player's trajectory is stochastic due to the randomness of his own Brownian, the population's trajectory is deterministic conditional on the state. In game theoretic terminology, this assumption is called \emph{no aggregated uncertainty}. Note, however, that in our environment, the population dynamic is still random because it depends on the random state. The second item captures the idea that the players take the population dynamic as given and ignore their own impact on it. In game theoretic terminology, players are \emph{outcome takers}~\cite{green1984continuum}. 

Both of this assumptions make sense when the number of players is large. Indeed, the MFG equilibrium, sometimes called \emph{a non-atomic equilibrium} in game theoretic literature~\cite{schmeidler1973equilibrium,mas1984theorem}, approximates a game with many players as follows: Consider a game with $N$ players, each player observes a signals and controls his positions according to~\eqref{dzt} and~\eqref{dxt} where the Brownian components are independent between players. The flow cost to the players depends on the empirical distribution of the players' positions $\frac{1}{N}\sum_i\delta_{x^i_t}$ where $x^i_t$ is the position of player $i$ at time $t$.  The MFG equilibrium strategies given by the control $\alpha$ are $\epsilon$-Nash equilibrium in the finite player game for sufficiently large $N$.

As we mentioned, most of the game-theoretic literature focuses on discrete-time games. The continuous-time modeling is also an approximation of discrete-time players when players take action and receive signals in short intervals, or \emph{frequently} in game theoretic terminology~\cite{sannikov2010role}. More explicitly, in the \emph{discrete-time game with round duration $\Delta$}, at the beginning of each period each player receives a signal from distribution $\caln(s,\sigma^2\Delta)$, where $s$ is the realization of the state, and  signals are i.i.d.\ over players and rounds.  At the beginning of every round, each player is at some position $x_k$, chooses a control $\alpha_k$ and his position moves to $x_{k+1}=x_k+\alpha\Delta+\caln(0,\sigma'^2\Delta)$. At the end of the round, the player receives payoff $\Delta\left(C(S,x_k,\alpha_k)+F(S,x_k,\rho_k)\right)$ where $\rho_k$ is the empirical distribution of positions of all players at round $k$. The game is played for $T/\Delta$ rounds. The MFG equilibrium becomes a standard $\epsilon$-Nash equilibrium in the discrete game for sufficiently small $\Delta$.  
\subsection{Applications} \label{sec:app}
%\st{We do not pretend that our model is directly applicable, given the strong assumptions on the position space and the information structure. Rather, our contribution is to provide a tractable theoretical model that can serve as a}
Our model is a basis for future theoretical developments and applications, including the situations that we mention now. 
\begin{itemize}
\item
\textit{Product differentiation.} consider firms that sell the same type of product (e.g. cellulars). They have a time period $T$ to develop their product. Their position represents the characteristics of the product (size, quality, price...). The state $S$ represents the ideal characteristics of the consumer, and only affects the terminal cost. At time $T$, the product is put on the market. Their aim is that the product characteristics are not far from the ideal of the consumer, and moreover, that their product is far from the one of the other firms. The terminal cost reflects these preferences, and can be taken for instance as $G(s,x,\rho)=|x-s|^2-\int |x-y|^2 \rho(\diff y)$. The flow cost is only a function of the position and the control.
\item
\textit{Portfolio management.} 
Consider investors that put funds in some investment scheme during some time period, (life insurance, hedge fund), such that the return interest rate is only known at the end of the period. The state $S$ represents the interest rate, and the position of each player represents the composition of their portfolio. 

\end{itemize}

\section{Mean field game equations} \label{sec:system}
\subsection{The HJB equation}
Fix the population density $\rho_{s,t}$ over $\m{T}^n$ at time $t$ if the state is $s$. In this section we derive the solution to the single-player control problem from Section~\ref{se:control}.

The tractability of our model relies on the fact that  the aggregated signal $Z_t$ at time $t$ is a sufficient statistic for $S$ given the signals $Z_\tau$ at times $\tau\in [0,t]$. This sufficiency, combined with the fact that the evolution of position $X_t$ in~\eqref{dxt} depends on the signal but not directly on the state, imply that the signal $Z_t$ at time $t$ is also a sufficient statistic for $S$ given the process $(Z_\tau,X_\tau)$ at times $\tau\in [0,t]$. This allows us to reduce the player's problem to a standard control problem with the state variable $(z,x)$.

We first describe the player's evolution of beliefs about the state, which is a standard argument in Bayesian statistics. We denote by $\caln(\mu,\sigma^2)$ the normal law of mean $\mu$ and variance $\sigma^2$, and by $\varphi_{\mu,\sigma^2}$ its probability density function.
\vspace{1cm}
\begin{lemma}\label{le:bayesian}
For all $t \in \m{R}_+$,
\begin{enumerate}[(i)]
\item
  The law of $Z_t|S=s$ is $\caln(st,\sigma^2 t)$
  \item The law of $Z_t$ is $\caln(0,\sigma^2t+t^2)$.
\item
  The law of $S|Z_t=z$ is $\caln(r_t(z),{\sigma_t}^2)$, where 
  \begin{equation}
r_t(z)=\frac{z}{t+\sigma^2} \quad \text{and} \quad \sigma_t^2=\frac{\sigma^2}{\sigma^2+t}
\end{equation}
\end{enumerate}
\end{lemma}
\begin{proof}
The first item follows from~\eqref{dzt}. The second and third items follow from the fact that
\begin{equation}\label{bayes}\varphi_{0,1}(s)\varphi_{st,\sigma^2t}(z)=\varphi_{0,\sigma^2t+t^2}(z)\varphi_{r_t(z), \sigma_t^2}(s).\end{equation}\end{proof}
Recall that our cost functions $F,C,G$ depend on the state that is not observed by the players.  
Let $\tilde F(t,z,x,\rho), \tilde C(t,z,x,\alpha), \tilde G(z,x,\rho)$ be the player's corresponding expected cost given the signal $z$ at time $t$. More explicitly
\begin{equation}\label{costtilde}\begin{split}&\tilde F(t,z,x,\rho)=\int \varphi_{r_t(z), \sigma^2_t}(s) F(s,x, \rho_{s,t})\diff s,\\&\tilde C(t,z,x,\alpha)=\int \varphi_{r_t(z), \sigma^2_t}(s) C(s,x, \alpha)\diff s,\text{ and}\\&\tilde G(z,x,\rho)=\int \varphi_{r_T(z), \sigma^2_T} (s)G(s,x, \rho_{s,T})\diff s\end{split}\end{equation}

By replacing the integrand in~\eqref{cost} with its conditional expectation at time $\tau$ and then using the sufficiency of the signal process we can rewrite the player's cost function as
\begin{equation}\label{cost-no-s}
\m{E}  \int_0^T \left(\tilde C(t,Z_t,X_t,\alpha_t)+\tilde F(t,Z_t,X_t,\rho_{S,t}) \diff t \right) + \tilde G(t,Z_t,X_T,\rho_{S,T}),
\end{equation}

In addition, it follows from Lemma~\ref{le:bayesian} that the signal process of a single player evolves according the  SDE.  
\begin{equation}\label{dzt-no-s}\td Z_t = r_t(Z_t)\td t + \sigma^2\td B_t.\end{equation}
Therefore, the control problem given by~\eqref{cost-no-s} is  a standard control problem, with state $z,x$ and value function
\begin{equation}\label{udef}
u(t,z,x)=\inf_{\alpha_t}\bbe \int_t^T \left(\tilde F(r,Z_r,X_r,\rho_{S,t})+\tilde C(r,Z_r,X_r,\alpha_t)\right)\diff r + \tilde G(T,Z_T,X_T,\rho_{S,T})\end{equation}
with $X_t=x$ and $Z_t=z$ and evolution given by ~\eqref{dxt} and~\eqref{dzt-no-s}.

For $(t,z,x) \in \m{R}_+ \times \m{R}^n \times \m{R}^m$
and $p \in \m{R}^n$, let
\begin{equation}\label{hdef}
H(t,z,x,p):=\sup_{\alpha}\left\{-\alpha \cdot p - \tilde C(t,z,x,\alpha)\right\}
\end{equation}
The Hamilton-Jacobi-Bellman equation for $u$ is
\begin{equation}\label{hjb}
  -D_tu+H(t,z,x,D_x u)-D_zu \cdot r_t(z)-\frac{\sigma^2}{2}\Delta_z u-
  \frac{{\sigma'}^2}{2} \Delta_x u=\tilde F(t,z,x,m),\end{equation}
  with the terminal condition  
\begin{equation}\label{hjb-terminal}u(T,z,x)= \tilde G(z,x).\end{equation}
Standard verification results (for example~\cite[Lemma 3.1.6]{cardaliaguet10}) imply that if $u(t,x,z)$ is a classical solution to~\eqref{hjb} then $u$ is  the value function of~\eqref{udef} and the optimal control at $(t,z,x)$ is $-D_pH(t,z,x,D_xu)$.
\subsection{The Fokker-Planck Equation}
Assume each agent chooses the control $\alpha(t,z,x)$, which is continuous in $t$ and Holder continuous in $z,x$. Then for every $s\in \bbr$ the SDE
\begin{equation}\label{forwardsde}\begin{split}
\td Z_t=s dt+\sigma dB_t
\td X_t=\alpha_t dt+\sigma'dB'_t.\end{split}
\end{equation}
has a unique solution and the density $m(s,t,z,x)$ of $(Z_t,X_t)$ at time $t$ is a weak solution for the Fokker-Planck equation (See~\cite[Lemma 3.1.3]{cardaliaguet10})
\begin{equation}\label{fp}
  \partial_t m +\div_x(\alpha(t,z,x) m)+s\cdot \div_zm-\frac{\sigma^2}{2}\Delta_{z}m-  \frac{{\sigma'}^2}{2}\Delta_xm=0
\end{equation}
with the initial condition
\begin{equation}\label{fp-initial}m(s,0,z,x)=\bar\rho(x)\delta_0(z).\end{equation}Here $\bar\rho\in\Delta(\m{T}^n)$ is the initial density of positions, and we assume that the signal process starts with $0$ for all players (See~\cite[Section 3.1.2]{cardaliaguet10})

Recall that the signal $Z_t$ at time $t$ is  a sufficient statistic for $S$ given the process $(Z_\tau,X_\tau)$ at times $\tau\in [0,t]$. The following lemma establishes the implication of sufficiency on the solution to the Fokker-Planck equations~\eqref{fp}: The solution factorizes to two terms, similar to the classical Fisher–Neyman factorization theorem.  In the lemma, $\tau$ does not depend on $s$. Note that $Z_t/t$ is an unbiased estimator for $S$, which explains the drift coefficient that multiplies $div_z\tau$ in the PDE of $\tau$.
We will use the lemma in our proofs of existence and uniqueness of solutions to the MFG system.
\vspace{1cm}
\begin{lemma}\label{le:fp}
The solution to~\eqref{fp} is of the form $m(s,t,z,x)=\varphi_{st,\sigma^2 t}(z)\tau(t,z,x)$ where $\tau(t,z,x)$ is the solution to
\[\partial_t\tau+\div_x(\alpha(t,z,x)\tau)+\frac{z}{t}\div_z\tau-\frac{\sigma^2}{2}\Delta_{z}\tau-  \frac{{\sigma'}^2}{2}\Delta_x\tau=0,\]
with the initial condition
\[\tau(0,z,x)=\bar\rho(x).\]
\end{lemma}
\begin{proof}
Note first that $\varphi_{st,\sigma^2t}$ is the solution to
\[\partial_t\varphi+s\div_z\varphi-\frac{\sigma^2}{2}\Delta_z\varphi=0,\]
with the initial condition
\[\varphi(0,z)=\delta_0(z) .\]
Assume $\varphi$ and $\tau$ satisfy these equations and let $m=\varphi\cdot\tau$. Then
\[\div_x \alpha m=\varphi\div_x\alpha\tau,\text{ and        }\div_zm=\tau\div_z\varphi+\varphi\div_z\tau\]
and
\[\Delta_xm=\varphi\Delta_x\tau\text{ and                  }\Delta_zm=\tau\Delta_z\varphi+2\div_z\varphi\div_z\tau+\varphi\Delta_z\tau\]
Therefore
\begin{align*}
\partial_tm=&\tau\partial_t\varphi+\varphi\partial_t\tau=\\&
\tau(-s\div_z\varphi+\frac{\sigma^2}{2}\Delta_z\varphi)+\varphi(-\div_x(\alpha\tau)-\frac{z}{t}\div_z\tau+\frac{\sigma^2}{2}\Delta_z\tau+  \frac{{\sigma'}^2}{2} \Delta_x\tau)=\\
            &
-\div_x(\alpha m)-s\div_z(m)+\frac{\sigma^2}{2}\Delta_{z}m+  \frac{{\sigma'}^2}{2}\Delta_xm.\end{align*}
where the last equality uses the fact that 
\[\sigma^2\div_z\varphi=(s-\frac{z}{t})\varphi.\]
\end{proof}
\subsection{The MFG system}
\begin{definition}
An \emph{equilibrium} is a pair $(u,m)$ that satisfies the following system in the classical sense: 
$$
 \quad  \left\{
    \begin{array}{ll}
        &   -D_tu+H(t,z,x,D_x u)-D_zu \cdot r_t(z)-\frac{\sigma^2}{2}\Delta_z u-
  \Delta_x u=\tilde F(t,z,x,m)  \\
  &     \partial_t m +\div_x(-D_p H(t,z,x,D_xu) m)+s\cdot \div_zm-\frac{\sigma^2}{2}\Delta_{z}m-  \frac{{\sigma'}^2}{2}\Delta_xm=0
  \\
  & u(T,z,x)= \tilde G(z,x)
  \\ & m(s,0,z,x)=\bar\rho(x)\delta_0(z)
    \end{array}
\right.
$$
\end{definition}
\begin{remark}
The above equations do not constitute a standard MFG system. Indeed, the parameter $s$ appearing in the Fokker-Planck equation is not known to the players. An interesting question is whether such a system can be rewritten with a new set of variables, that are known to the players. This type of transformation is classical in discrete-time zero-sum games,and the new state variables incorporate the \textit{beliefs} of players about the unknown parameters of the game  (see \cite{astrom65} for the 1-Player case and \cite[Chapter IV]{MSZ} for the 2-Player case). In the MFG context, this has been exploited in \cite{bertucci22,BCR23}. In our setting, such variables could be the triple $(z,x,\mu)$, where $\mu$ is a probability measure representing the belief of a player over the set of population distributions. This would yield to a MFG system written on an infinite dimensional state space. 
%This system may be quite helpful to analyze general information structures, but in our normal-distribution setting, the above system is more handy to analyze.
\end{remark}
%\textcolor{blue}{I am not sure about this remark. First, isn't this reduction to a stochastic game only valid to zero-sum games? Second, even if it can be done in n player games the nonatomic players seem like another step}
\subsection{Comparison to previous models}
We now compare our model to the two standard models of mean-field games, first without common noise and then with common noise. 

Recall our players' signaling structure~\eqref{dzt}. For two distinct players $i,j$ their signaling processes $\td Z_t^i$ and $\td Z_t^j$ are stochastically dependent since their drift is governed by the same random variable $S$. Under fixed control functions $\alpha^i(t,z,x)$ and $\alpha^j(t,z,x)$ of the players, their positions $X_t^i$ and $X_t^j$ are therefore also stochastically dependent. 

The canonical model of mean-field games without common noise can be viewed as our model with a fixed state. In this model, unlike our model, for a fixed control functions, the positions of two players are stochastically independent because they are governed by independent noises. 

In stochastic games with common noise the positions of the players is governed by idiosyncratic brownian motion and also by some common brownian noise: If we denote by $Y_t^i$ and $Y_t^j$ the positions of two players under fixed controls  $\alpha^i(t,y)$ and $\alpha^j(t,y)$ then 
\[\td Y_t^i=\alpha^i(t,Y_t^i)+\td B_t^i+\td B_0^i\]
where $B_t^i$ are the idiosyncratic noises and $B_t^0$ is the common noise. 
Because of the common noise, the players' positions are also stochastically dependent, as in our model. However, the nature of the dependence is different: For a fixed $\bar t$, the future process $(Y^i_t)_{t\ge \bar t}$ of player $i$ is conditionally independent of the history $(Y^j_t)_{t\le \bar t}$ of player $j$ given the history $(Y^i_t)_{t\le \bar t}$ of player $i$. In other words, player $i$, who knows his history up to time $\bar t$, would not learn anything new about his future from observing the history of his opponent. This property does not hold in our model, since player $i$ could gain more information about the state $S$ by observing the trajectory of player $j$. 

The difference between our model and these two canonical model can also be seen from the structure of the mean field equations. In the canonical model without common noise there is a single Fokker-Planck equation. We have a paramatrized set of Fokker-Planck equation and their joint solution has to be compatible with the HJB equation through the expected cost function $\tilde F$ in~\eqref{costtilde}. Note that for a fixed $s$, our $m_s$ is not a solution to the corresponding HJB equation with cost $F$, hence our system does not reduce to standard MFG equations. The forward equation of the canonical model with common noise is a stochastic equation governed by the common noise component, while we have a parametrized family of standard Fokker Plank partial differential equations.

In motivation and modeling approaches, our paper is probably most similar to Bertucci's model \cite{bertucci22} of mean field games with incomplete information. In his model, the initial distribution of players positions is random, unknown to the players, and plays a similar role to our unknown state $S$. In particular, his HJB equation also involves an expected cost function under players' belief at that time. In addition to the difference in modeling the information flow, the results are different: In our model uniqueness is obtained under standard monotonicity assumptions on the primitives, while Bertucci shows that this does not hold in his model. 
  
  \section{Existence} \label{sec:existence}
In this section, we provide sufficient conditions under which the MFG system has a classical solution, that are largely inspired by \cite{cardaliaguet10}.
\begin{enumerate}
\item The functions $F$ and $G$ are continuous on $\bbr\times\m{T}^d\times \Delta(\m{T}^d)$.
\item The functions $F(s,\cdot,m)$ and $G(s,\cdot,m)$ are in $C^{1+\beta}(\m{T}^d)$ and $C^{2+\beta}(\m{T}^d)$ for some $\beta\in (0,1)$, with norm that is bounded by a sub-exponential function of $s$, uniformly w.r.t. $m$.
\item The function $F(s,\cdot,\pi)$ and $G(s,\cdot,\pi)$ are bounded in $C^{1+\beta}(\m{T}^d)$ and $C^{2+\beta}(\m{T}^d)$ (for some $\beta\in (0,1)$) uniformly in $s$ and $\pi$.
\item \label{hyp_ex_3}
The cost function $C$ is continuously differentiable in $x,\alpha$ and $C(s,x,\alpha)$ has exponential growth in $s$, and $\alpha \rightarrow C(s,x,\alpha)$ is strongly convex, uniformly in $s$ and $x$. 
%The cost function $C$ is twice continuously differentiable, and $\alpha \rightarrow C(s,x,\alpha)$ is strongly convex, uniformly in $s$ and $x$. Then, the function $\tilde{C}$ enjoys the same properties. 
%Consequently, $H$ is locally Lipschitz in $(t,z,x)$, and $p \rightarrow H(.,p)$ is twice differentiable and has a Lipschitz gradient, where the Lipschitz constant is uniform in $(t,z,x)$.
%\\
%\textcolor{red}{We probably do not need $C$ to be $C^2$: we only need local Lipschitz, 
%that $C$ is differentiable with respect to $p$, and that $D_pH$ is differentiable in $(x,p)$. But it could be safer to make this stronger assumption}. 
%\textcolor{blue}{We need $\tilde C$ to be locally Lipschitz so we need $C$ to be locally L and some condition on the Lipschitz constant being integrable w.r.t to normal distribution. Not sure if $C^2$ is helpful here}
  \item \label{hyp_ex_4}
  The cost function satisfies the following growth condition: there exists $A>0$, for all $(s,x,\alpha)$,
  \begin{equation*}
  D_x C(s,z, x,\alpha) \cdot p \geq -A (1+|\alpha|^2)
  \end{equation*}
%    Together with \ref{hyp_ex_3}, it implies a similar growth condition on $H$: there exists $B>0$, for all $(s,x,\alpha)$, 
%    \begin{equation*}
%  D_{z,x} H(t,z,x,\alpha) \cdot p \geq -B (1+|p|^2).
%   \end{equation*}
  \end{enumerate}
    \begin{theorem}~\label{th:existence}
  Under Assumptions 1-4, the MFG system has a classical solution. 
  \end{theorem}
Before we prove Theorem~\ref{th:existence}, we explain how the assumptions on the cost function translate to properties of the Hamiltonian.
  The function $\tilde C$ enjoys the same properties of $C$. Consequently, the maximum $\alpha^\ast(t,z,x,p)$ is attained in~\eqref{hdef} and the Hamiltonian is differentiable w.r.t. $p$ and $x$ with
\begin{equation}\label{hdifs}\begin{split}
&D_p H = -\alpha^\ast(t,z,x,p)\\
&D_{z,x} H = -D_{z,x} \tilde C(t,z,x,\alpha^\ast(t,z,x,p))
\end{split}\end{equation}
In addition, $D_pH$ is globally Lipschitz, where the Lipschitz constant is uniform in $(t,z,x)$ and  depends only on the strong convexity constant of $C$~\cite{zhou2018fenchel}.  

The global Lipschitz property property on $D_pH$ and~\eqref{hdifs} implies that $\alpha^\ast(t,z,x,p)$ has a linear growth in $p$. Together with the growth condition on $\tilde C$ and~\eqref{hdifs} again it implies a  growth condition on $H$: there exists $B>0$, for all $(s,x,\alpha)$, 
    \begin{equation*}
  D_{z,x} H(t,z,x,\alpha) \cdot p \geq -B (1+|p|^2).
  \end{equation*}

  \begin{proof}
  Consider the space $C^0(\m{R}\times [0,T],\mathcal{P}(\m{T}^d))$ of all possible {position distribution} $\mu$, with the interpretation that $\mu(s,t)$ is the empirical distribution of players' positions at time $t$ if state is $s$.
The assumptions on $F$ and $G$ imply that $\tilde F(t,z,x,\mu)$ and $\tilde G(z,x,\mu)$ are bounded in $C^{1+\beta}$ and $C^{2+\beta}$ uniformly w.r.t. $\mu,t$ and that the maps $\mu \mapsto ((t,z,x) \rightarrow \tilde{F}(t,z,x,\mu))$ and $\mu \mapsto ((z,x) \rightarrow \tilde{G}(z,x,\mu))$ from $C^0(\m{R}\times [0,T],\mathcal{P}(\m{T}^d))$ to $C^0([0,T] \times \m{T}^d \times \m{R})$ are continuous.

For a large $C>0$ and, let $\mathcal{C}$ be the set of elements $\mu$ in $C^0(\m{R}\times [0,T],\mathcal{P}(\m{T}^d))$ such that
\begin{align}
&\label{holder-t}
d_1(\mu(s,t),\mu(s,t')) \leq C |t-t'|^{1/2},\text{ and}\\
  &\label{holder-s}
    d_1(\mu(s,t),\mu(s',t)) \leq C |s-s'|^{1/2}.
\end{align}
Then $\mathcal{C}$ is a convex closed subset of $C^0(\m{R}\times [0,T],\m{T}^d)$, which we equip with the topology of compact convergence. It is compact, thanks to Arzela-Ascoli theorem. 

Let $\mu \in \mathcal{C}$. We associate $\nu=\Psi(\mu)$ in the following way: 
Let $u$ be the unique solution to 
% \begin{equation}\label{hjb}
%  -D_tu+H(t,z,x,D_x u)-D_zu \cdot r_t(z)-
%  \Delta_x u-\Delta_z u=\tilde F(t,z,x,m).\end{equation}
\begin{equation} \label{existence_HJB}
    \begin{cases}
       &  -D_tu+H(t,z,x,D_x u)-D_zu \cdot r_t(z)-
  \Delta_x u-\Delta_z u=\tilde F(t,z,x,\mu)\\
      & u(z,x,T)=\tilde{G}(t,z,x,\mu(T))  \\
      %0 & \text{otherwise}
    \end{cases}       
\end{equation}
The existence of a solution follows from~\cite[Corollary A.2.2]{cardaliaguet10}, the growth condition on $H$ and the  regularity of $\tilde F(t,z,x,\mu)$ and $\tilde G(z,x,\mu)$.
Then we define $m: S \times [0,T] \times \m{R}^n \times \m{R}$ such as for any $s \in \m{R}$, $m(s,.)$ is the solution of the Fokker-Planck equation 
\begin{equation}
    \begin{cases}
       &    \partial_t m +\div_x(-D_pH(t,z,x,D_xu) m)+s\cdot \div_zm-\frac{\sigma^2}{2}\Delta_{z}m-  \frac{{\sigma'}^2}{2} \Delta_xm=0 \\
      & m(.,0)=m_0  \\
      %0 & \text{otherwise}
    \end{cases}       
\end{equation}
Last, we define $\nu$ such that: for any $s,t,x$, $\nu(s,t,x)$ is the marginal of $m(s,t,x,.)$. 
We now claim that $\nu\in \mathcal{C}$. The global Lipschitz assumption on $H$ and the classical $\frac{1}{2}$-Holder estimate on the Fokker Plank equation \cite[Lemma 3.1.4]{cardaliaguet10} imply~\eqref{holder-t}. Lemma~\ref{le:fp} implies~\eqref{holder-s}.

Finally, the maps $\Psi$ is continuous. Indeed, let $\mu_n\in \mathcal{C}$ converge to $\mu_\infty$ and let $(u_n,m_n)$ be the corresponding solutions. Then the maps $(t,z,x)\mapsto F(t,z,x,\mu_n)$ and $(t,z,x)\mapsto G(t,z,x,\mu_n)$ converge uniformly to $(t,z,x)\mapsto F(t,z,x,\mu_\infty)$ and $(t,z,x)\mapsto G(t,z,x,\mu_\infty)$ thanks to the continuity properties of $\tilde F$ and $\tilde G$. Moreover, as the RHS of the HJ for $u_n$ is bounded in $C^{1+\beta,1+\beta/2}$, then the $u_n$ are uniformly bounded in $C^{2+\beta,1+\beta/2}$, hence converges in $C^{2,1}$ to the unique solution $u_\infty$ of the HJB with RHS $F(\cdot,\mu_\infty)$. Then the $m_n$ are solutions of a linear equation with uniformly Holder continuous coefficient, hence are uniformly bounded in $C^{2+\beta,1+\beta/2}$. 
Then the $m_n$ converge in $C^{1,2}$ to the unique solution of the FP equation corresponding to $u_\infty$. This implies converges of $\nu_n$ to $\nu_\infty$ in $\mathcal{C}$.

By Schauder fixed point Theorem the map $\Psi$ admits a fxed point $\mu$. The corresponding $u,m$ are an equilibrium.

           %            The assumptions on $F$ and $G$ imply $\tilde F(t,z,x,\mu)$ is smooth in $z,x$ and that $\phi: \mu \rightarrow ((t,z,x) \rightarrow \tilde{F}(t,z,x,\mu))$ from $\mathcal{C}$ to $C^{1,1}([0,T] \times \m{T}^d \times \m{R})$ is well-defined and continuous. The same holds for the mapping $\mu \rightarrow ((z,x) \rightarrow \tilde{G}(z,x,\mu))$ from $\mathcal{C}$ to $C^{2,1}([0,T] \times \m{T}^d \times \m{R})$.

Together with the assumptions on $H$, this implies that the Hamilton-Jacobi equation (\ref{existence_HJB}) has a classical smooth solution. %solution in $C^{2,1}$
\end{proof}

  \section{Uniqueness} \label{sec:uniqueness}
  We made the following assumptions: 
  \begin{enumerate}
  \item
  The function $\alpha \rightarrow C(.,\alpha)$ has a Lipschitz gradient, where the Lipschitz constant is uniform in $(s,x)$. This implies that $H$ is strongly convex in $p$, uniformly in $(t,z,x)$. 
  %We assume that $C$ is smooth, convex in $\alpha$, 
  %and that there exists $c>0$, for all $(s,x,\alpha,\alpha')$, 
%  \begin{equation*}
%  \left| \nabla C(s,x,\alpha)-\nabla C(s,x,\alpha') \right| \leq c \left\| \alpha-\alpha' \right\|
%  \end{equation*}
%  This implies that for all $(t,z,x,\alpha,\alpha')$, 
%    \begin{equation*}
%  \left| \nabla \tilde{C}(s,x,\alpha)-\nabla \tilde{C}(s,x,\alpha') \right| \leq c \left\| \alpha-\alpha' \right\|
%  \end{equation*}
%  
%  This implies that the Hamiltonian is $(1/c)$-strongly convex in $p$: for all $(p,p')$, 
%   \begin{equation*}
%   H(t,z,x,p) \geq H(t,z,x,p') +\nabla H(t,z,x,p') \cdot (p-p')+\frac{1}{c} \left\|p-p'\right\|^2
% \end{equation*}
%  
%We assume that a uniform convexity bound on the cost function:
%\begin{equation}\label{uniform-c}
%\underbar c I_n\le D^2_{\alpha\alpha}C(s,x,\alpha) \le \bar c I_n
%\end{equation}
%for some $0 < \underbar c < \bar c$. 
%This implies a uniform convexity bound on the expected cost
%\begin{equation*}
%\underbar c I_n\le D^2_{\alpha\alpha}\tilde C(t,z,x,\alpha) \le \bar c I_n
%\end{equation*}
%which implies a uniform convexity bound on the Hamiltonian
%\begin{equation*}
%\frac{1}{\bar c} I_n\le D^2_{pp} H(t,z,x,p) \le \frac{1}{\underbar c} I_n
%\end{equation*}
\item
We assume that $F$ satisfies the monotonicity condition \cite{LL07} for every $s$:
 the cost functions $F$ and $G$ satisfy
\[\begin{split}&\int \left(\rho^1(x)-\rho^2(x)\right)\left(F(s,x,\rho^1)-F(s,x,\rho^2)\right)\diff x \geq 0,\text{and}\\
&\int \left(\rho^1(x)-\rho^2(x)\right)\left(G(s,x,\rho^1)-G(s,x,\rho^2)\right)\diff x \geq 0.\end{split}\]
  for every $s\in\bbr^m$ and every $\rho^1\neq \rho^2\in \calp(\bbr^n)$. 
\end{enumerate}
\begin{theorem}Under this assumption, if $(u^1,m_1)$ and $(u^2,m_2)$ are two solutions. Then $u^1=u^2$ and $m_1(s,t,x) z=m_2(s,t,x)$ for almost every $s,t$. \end{theorem}
\begin{proof}
Fix $s \in \m{R}$. Set $\bar{u}=u_1-u_2$ and $\bar{m}=m_1-m_2$. Then
\begin{eqnarray*}
\frac{d}{dt} \int_{\m{T}^d \times \m{R}} \bar{u} \bar{m}&=&\int_{\m{T}^d \times \m{R}} (\partial_t \bar{u}) \bar{m}+\bar{u} (\partial_t \bar{m})
\\
&=& \int_{\m{T}^d \times \m{R}} (-\Delta_{z,x} \bar{u} +H(t,z,x,D_x u_1)-H(t,z,x,D_x u_2)-D_z \bar{u} \cdot r_t(z)-\tilde{F}(t,z,x,m_1)+\tilde{F}(t,z,x,m_2)) \bar{m}
\\
&+& \int_{\m{T}^d \times \m{R}} \bar{u} \Delta_{z,x} \bar{m} -\langle D_x\bar{u},m_1 D_p H(t,z,x,D_xu_1)-m_2 D_pH(t,z,x,D_x u_2)\rangle+s \bar{m} \cdot D_z\bar{u},
\end{eqnarray*}
where the last term is obtained after integration by parts. 
Note that
\[\int_{\m{T}^d \times \m{R}} -(\Delta\bar u)\bar m+\bar u(\Delta\bar m)=0.\]

We now rewrite the remaining terms in $H$ in the following way
\begin{align*}
\int_{\m{T}^d \times \m{R}} & (H(t,z,x,D_x u_1)-H(t,z,x,D_x u_2))\bar m-\langle D_x\bar{u},m_1 D_p H(t,z,x,D_xu_1)-m_2 D_pH(t,z,x,D_x u_2)\rangle=\\
&-\int_{\m{T}^d \times \m{R}}m_1\left( H(t,z,x,D_x u_2)-H(t,z,x,D_x u_1)-\langle D_pH(t,z,x,D_xu_1),D_xu_2-D_xu_1\rangle\right)\\
&-\int_{\m{T}^d \times \m{R}}m_2\left( H(t,z,x,D_x u_1)-H(t,z,x,D_x u_2)-\langle D_pH(t,z,x,D_xu_2),D_xu_1-D_xu_2\rangle\right)\\
&\le -\int_{\m{T}^d \times \m{R}}\frac{m_1+m_2}{2C}|D_xu_1-D_xu_2|^2.
\end{align*}
where the inequality follows from the uniform convexity assumption on $H$.

We now claim that from the assumption on $F$ it follows that
\[\int \varphi_{0,1}(s)\diff s\int_{\m{T}^d \times \m{R}}  \bar m(s,t,z,x) (\tilde F(t,z,x,m_1)-\tilde F(t,z,x,m_2))\diff s\diff x\diff z\ge 0\]
For every $t$. 
Indeed, by Lemma~\ref{le:fp} $m_i$ have the form $m_i(s,t,z,x)=\varphi_{st,\sigma^2 t}\tau_i(x|z)$ and therefore \begin{equation}\label{barm}\bar m(s,t,z,x)=\varphi_{st,\sigma^2 t}\bar\tau(x|z)\end{equation} with $\bar \tau=\tau_1-\tau_2$. Therefore,
\[\begin{split}&\int \varphi_{0,1}(s) \bar m(s,t,z,x) (\tilde F(t,z,x,m_1)-\tilde F(t,z,x,m_2))\diff s\diff x\diff z=\\    
%    &\int \varphi_{0,1}(s) \bar m(s,t,z,x) \left(\int_{s'}\varphi_{r_t(z), \sigma^2_t}(s')\left(F(x,\hat m_1(s',t,\cdot),s')-F(x, \hat m_2(s',t,\cdot),s')\right)\diff s'\right) \diff s\diff x\diff z=\\
&\int \varphi_{0,1}(s) \bar m(s,t,z,x) \varphi_{r_t(z), \sigma^2_t}(s')\left(F(x,\hat m_1(s',t,\cdot),s')-F(x, \hat m_2(s',t,\cdot),s')\right)\diff s'\diff s\diff x\diff z=\\
%    &\int\varphi_{0,1}(s)\varphi_{st, \sigma^2  t}(z)\bar\tau(x|z) \left(\int_{s'}\varphi_{r_t(z), \sigma^2_t}(s')\left(F(x, \hat m_1(s',t,\cdot),s')-F(x, \hat m_2(s',t,\cdot),s')\right) \diff s'\right) \diff s \diff x \diff z=\\
    &\int\varphi_{0,1}(s)\varphi_{st, \sigma^2  t}(z)\bar\tau(x|z) \varphi_{r_t(z), \sigma^2_t}(s')\left(F(x, \hat m_1(s',t,\cdot),s')-F(x, \hat m_2(s',t,\cdot),s')\right) \diff s' \diff s \diff x \diff z=\\
    &\int\varphi_{0,\sigma^2 t+t^2}(z) \varphi_{r_t(z), \sigma^2_t}(s')\bar\tau(x|z) \left(F(x, \hat m_1(s',t,\cdot),s')-F(x, \hat m_2(s',t,\cdot),s')\right) \diff s'\diff x\diff z=\\
    &\int \varphi_{0,1}(s')\varphi_{s't,\sigma^2 t}(z)\bar\tau(x|z) \left(F(x, \hat m_1(s',t,\cdot),s')-F(x, \hat m_2(s',t,\cdot),s')\right) \diff s'\diff x\diff z=\\
    &\int \varphi_{0,1}(s') \diff s'\left(\int \left(\hat m_1(s',t,x)-\hat m_2(s',t,x)\right)\left(F(x, \hat m_1(s',t,\cdot),s')-F(x, \hat m_2(s',t,\cdot),s')\right)\diff x\right)\ge 0\end{split}\]
The first equality follows from~\eqref{costtilde}. The second from~\eqref{barm}. The third and fourth from~\eqref{bayes}. The fifth from~\eqref{barm}. The inequality follows from the state-wise monotonicity assumption. 

Similarly, from the monotonicity assumption on $G$ it follows that
\begin{equation}\label{tildeg}
\int_{\bbr\times \m{T}^d \times \m{R}} \varphi_{0,1}(s) \bar m(s,t,z,x) (\tilde G(z,x,m_1)-\tilde G(z,x,m_2))\diff s\diff x\diff z\ge 0.\end{equation}

Finally, for the remaining terms with $D_z\bar u$
\[\int \varphi_{0,1}(s) \bar m(s,t,z,x) D_z\bar u\left(-r_t(z)+s \right)=\bbe \int_x \tau(x|t,Z)D_Z\bar u(t,z,x)\cdot (S-\bbe(S|Z))=0\]
where the expectation is taken w.r.t. $S\sim\caln(0,1)$ and $Z|S=s\sim \caln(s,\sigma^2 t)$.
%\\
%\textcolor{red}{write this in the same way as before?}
%A proof of the same assertion using integrals:
%\begin{align*}
%&\int \varphi_{0,1}(s) \bar m(s,t,z,x) D_z\bar u\left(-r_t(z)+s \right)=
%&\int\varphi_{0,1}(s)\varphi_{s,\sigma^2t}(z)\bar m(s,t,z,x) D_z\bar %u(t,z,x)\cdot(s-\int_{s'}

%\bbe \int_x \tau(x|t,Z)D_Z\bar u(t,z,x)\cdot (-\bbe(S|Z)+S)=0\]

Putting all the estimates together we get 
\[
\frac{\td}{\td t}\int\varphi_{0,1}(s)\td s\int_{\m{T}^d \times \m{R}} \bar{u} \bar{m}\le - \int\varphi_{0,1}(s)\td s\int_{\m{T}^d \times \m{R}}\frac{m_1+m_2}{2C}|D_xu_1-D_xu_2|^2
\]
We integrate this inequality on the time interval $[0,T]$ to obtain
\begin{equation}\label{weintegrate}\int\varphi_{0,1}(s)\td s\int_{\m{T}^d \times \m{R}} \bar u(T)\bar m(T)-\bar u(0)\bar m(0)\le -\int_0^T\int_{\m{T}^d \times \m{R}}\varphi_{0,1}(s)\frac{m_1+m_2}{2C}|D_xu_1-D_xu_2|^2\end{equation}
Note that $\bar m(0)=m_1(0)-m_2(0)=0$ while, as $\bar U(T)=\tilde G(t,z,x,m_1(T))-\tilde G(t,z,x,m_2(T))$,
\[\int\varphi_{0,1}(s)\td s\int_{\m{T}^d \times \m{R}} \bar u(T)\bar m(T)=\int_{\m{T}^d \times \m{R}}\int_{\bbr\times \m{T}^d \times \m{R}} \varphi_{0,1}(s) \bar m(s,t,z,x) (\tilde G(z,x,m_1)-\tilde G(z,x,m_2))\diff s\diff x\diff z\ge 0\]
thanks to~\eqref{tildeg}. So the LHS of~\eqref{weintegrate} is nonnegative, while the RHS is nonpositive, which implies that both sides must vanish. Therefore $D_xu_1 = D_xu_2$ in $\{m_1>0\}\cup \{m_2>0\}$. 
As a consequence $m_2$
actually solves the same equation as $m_1$ (with the same drifts $D_pH(t,z,x,D_xu_1) = D_pH(t,z,x,D_xu_2)$: hence $m_1 = m_2$. Then, in turn, $u_1$ and $u_2$ solve the same HJB equation, so that $u_1 = u_2$.
\end{proof}
\section*{Acknowledgments}
The authors are grateful to Charles Bertucci and Pierre Cardaliaguet for valuable discussions that helped improve this paper.
This work was supported by the French Agence Nationale de la Recherche (ANR) under reference ANR-21-CE40-0020 (CONVERGENCE project). Part of this work was done during a 1-year visit of Bruno Ziliotto to the Center for Mathematical Modeling (CMM) at University of Chile in 2023, under the IRL program of CNRS. 
\bibliography{biblio}

\end{document}